\documentclass[12pt]{amsart}

\usepackage{times}
\usepackage[usenames,dvipsnames]{color}
\usepackage{tikz}
\usetikzlibrary{decorations.markings,patterns}
\usepackage{graphicx}
\usepackage[utf8]{inputenc}
\usepackage{enumitem}
\usepackage{hyperref}

\title[Dimensions of components of tensor products]
{Dimensions of components of tensor products \\ of the linear groups representations \\ with applications to Beurling-Fourier algebras}

\author {Beno\^\i{}t Collins}
\address{D\'epartement de Math\'ematique et Statistique, Universit\'e d'Ottawa,
585 King Edward, Ottawa, ON, K1N6N5 Canada
\newline
\indent  
CNRS, Institut Camille Jordan Universit\'e  Lyon 1, 43 Bd du 11 Novembre 1918, 69622 Villeurbanne, 
France} 

\email{collins@uottawa.ca}

\author {Hun Hee Lee}
\address{Department of Mathematics, Chungbuk National University, 410 Sungbong-Ro, Heungduk-Gu, Cheongju 361-763, Korea} 

\email{hhlee@chungbuk.ac.kr}

\author {Piotr \'Sniady}
\address{Zentrum Mathematik, M5,
Technische Universität München, \linebreak
Boltzmannstrasse 3,
85748 Garching, Germany \newline \indent
Institute of Mathematics, Polish Academy of Sciences, \linebreak
\mbox{ul.~\'Sniadec\-kich 8,} 00-956 Warszawa, Poland
 \newline
\indent 
Institute of Mathematics,
University of Wroclaw,  \linebreak \mbox{pl.\ Grunwaldzki~2/4,} 50-384
Wroclaw, Poland
} 
\email{piotr.sniady@tum.de, piotr.sniady@math.uni.wroc.pl}

\theoremstyle{plain}
\newtheorem{lemma}{Lemma}[section]
\newtheorem{theorem}[lemma]{Theorem}
\newtheorem{proposition}[lemma]{Proposition}
\newtheorem{corollary}[lemma]{Corollary}

\theoremstyle{definition}

\theoremstyle{remark}

\DeclareMathOperator{\End}{End}
\DeclareMathOperator{\GL}{GL}
\DeclareMathOperator{\SL}{SL}
\DeclareMathOperator{\SU}{SU}
 
\DeclareMathOperator{\Det}{Det}

\newcommand{\C}{{\mathbb{C}}}
\newcommand{\R}{{\mathbb{R}}}
\newcommand{\Z}{{\mathbb{Z}}}

\newcommand{\Disc}{\mathcal{D}}
\newcommand{\Cont}{\mathcal{C}}
\newcommand{\indeks}{\mathcal{I}}
\newcommand{\G}{\mathcal{G}}
\newcommand{\GG}{\widehat{\mathcal{G}}}
\newcommand{\one}{\mathbf{1}}
\newcommand{\om}{\omega}

\newcommand{\norm}[1]{\left\Vert#1\right\Vert}

\DeclareMathOperator{\LIS}{LI}
\DeclareMathOperator{\vol}{vol}

\begin{document}

\setcounter{MaxMatrixCols}{15}

\begin{abstract}
We give universal upper bounds 
on the relative dimensions of isotypic components
of a tensor product of the linear group $\GL(n)$ representations 
and universal upper bounds on the relative dimensions of irreducible components
of a tensor product of the special linear group
$\SL(n)$ representations. 
This problem is motivated by harmonic analysis problems, and we give some
applications of this result in the theory of Beurling-Fourier algebras.
\end{abstract}

\keywords{representations of unitary groups, Kronecker tensor product of representations, Littlewood-Richardson coefficients, Young tableaux, 
Beurling-Fourier algebras}
\subjclass[2010]{%
05E10%Combinatorial aspects of representation theory
{} (Primary)
22E46%Semisimple Lie groups and their representations
,
43A30%Fourier and Fourier-Stieltjes transforms on nonabelian groups and on semigroups, etc.
,
47L30%Abstract operator algebras on Hilbert spaces
,
51F25%Orthogonal and unitary group
{} (Secondary)%
% % % % % % % 
% 20C30,%Representations of finite symmetric groups
% 17B22,%Root systems
% 20G05,%Representation theory
% 17B10,%Representations, algebraic theory (weights)
}

\maketitle

\section{Introduction}

\subsection{The main problem for linear groups $\GL(n)$}
In this paper we are interested in the following question: let $\lambda,\mu$ be two irreducible representations 
of the linear group $\GL(n)$ and consider the decomposition of their tensor product 
$\lambda\otimes\mu$ into \emph{isotypic} components. \emph{How big the dimension of such an 
isotypic component can be?}

For irreducible representations $\lambda,\mu,\nu$ we denote by $c_{\lambda,\mu}^{\nu}$ the 
\emph{Littlewood-Richardson coefficient}, i.e. the multiplicity of the irreducible representation $\nu$ 
in the Kronecker tensor product $\lambda\otimes \mu$. For an irreducible representation $\rho$ 
we denote by $d_\rho$ its dimension.
With these notations, the dimension of the isotypic component $[\nu]$ of $\lambda\otimes\mu$ is  
equal to $c_{\lambda,\mu}^{\nu} \ d_\nu$. Our goal will be to give an upper bound for the fraction
\begin{equation}
\label{eq:LR-measure}
P_{\lambda,\mu}(\nu):= \frac{c_{\lambda,\mu}^{\nu} \ d_\nu}{d_\lambda \ d_\mu}  
\end{equation}
which can be interpreted as the \emph{relative dimension} of the isotypic component $[\nu]$ in $\lambda\otimes\mu$.

Equation \eqref{eq:LR-measure} defines a probability distribution
 $P_{\lambda,\mu}$ (called the \emph{Littlewood-Richardson measure}) on irreducible representations. 
 This probability measure can be interpreted as a distribution of a random irreducible component of the 
 Kronecker tensor product $\lambda\otimes\nu$, where each irreducible component is sampled with a 
 probability proportional to its dimension. Our problem can be therefore equivalently formulated as finding 
 an upper bound for the atoms of Littlewood-Richardson measure.

\subsection{The main result for linear groups $\GL(n)$}
The main result of this paper is the following partial answer to the above problem.
\begin{theorem}
\label{theorem:main-theorem}
Let $n\geq 1$ be a fixed integer. There exists a constant $C_{n}$ such that for any
irreducible representations $\lambda,\mu,\nu$ of\/ $\GL(n)$ the atom of the Littlewood-Richardson measure is bounded from above as follows:
\begin{equation}
\label{eq:main-gln}
P_{\lambda,\mu}(\nu):=\frac{c_{\lambda,\mu}^{\nu} \ d_\nu}{d_\lambda \ d_\mu}  \leq C_n \left(\frac{1}{\lambda_1-\lambda_n}+\frac{1}{\mu_1-\mu_n} \right).
\end{equation}
\end{theorem}
Here, $\lambda_1\geq\cdots\geq \lambda_n$ and $\mu_1\geq\cdots\geq \mu_n$ are the components of the highest 
weight of $\lambda$ and $\mu$, respectively.
The notations used in the above inequality will be recalled in Section \ref{sec:representation-theory}. 
We postpone its proof to Section 
\ref{sec:proof-main-result}.
We will see that this result is optimal in a sense which will be clarified at the end of the paper. 

\subsection{The main result for special linear groups $\SL(n)$}
In this paper we are also interested in the following modification of the above problem: 
let $\lambda,\mu$ be two irreducible representations of the special linear group $\SL(n)$ and consider the decomposition of their tensor product $\lambda\otimes\mu$ into \emph{irreducible} components. \emph{How big the dimension of such an irreducible component can be?} 

A partial answer for this problem is given by the following result, which is a corollary to our main theorem:
\begin{corollary}
\label{cor:main-theorem-SLn}
Let $n\geq 1$ be a fixed integer. There exists a constant $C_{n}$ such that for any
irreducible representations $\lambda,\mu,\nu$ of\/ $\SL(n)$, if $\nu$ contributes (with multiplicity at least $1$) to the decomposition of the Kronecker tensor product $\lambda\otimes \mu$ into irreducible components, then its relative dimension is bounded from above as follows:
\begin{equation}
\label{eq:main-sln}
\frac{d_\nu}{d_\lambda \ d_\mu}  \leq C_n \left(\frac{1}{\lambda_1}+\frac{1}{\mu_1} \right). 
\end{equation}
\end{corollary}
The notations used in the above inequality will be recalled in Section \ref{subsec:sln}, where we will also present its proof.

\subsection{The case of unitary groups and special unitary groups}
The representation theory of the unitary group $U(n)$ is exactly the same as that of the linear group $\GL(n)$, namely the restriction map gives a one to one map; its inverse is given by the analytic continuation. 
In particular, the correspondence between irreducible representations and highest weights holds also for $U(n)$.
For this reason in the formulation of Theorem \ref{theorem:main-theorem} one can replace the 
representations of the linear groups $\GL(n)$ by the representations of the unitary groups $U(n)$ and 
the result holds true without any modifications.

Analogous relationship holds between the representation theory of the special unitary group $\SU(n)$ and the special linear group $\SL(n)$, for this reason in the formulation of Corollary \ref{cor:main-theorem-SLn} one can replace representations of $\SL(n)$ by representations of $\SU(n)$.

\subsection{Applications to Beurling-Fourier algebras}

Our paper is motivated by the work of Mahya Ghandehari, Hun Hee Lee, Ebrahim Samei and 
Nico Spronk \cite{GhandehariLeeSameiNico2012} and gives a proof of their conjecture 
(Conjecture 1, p.~19). Our main theorem implies that the conjecture is true for any integer 
$n\ge 2$, whilst it was proved for $2\le n \le 5$ in an elementary way in \cite{GhandehariLeeSameiNico2012}.

In this subsection we briefly describe what are Beurling-Fourier algebras and implications of our 
main results on them. See \cite{LS12,LST12} 
 for the details on Beurling-Fourier algebras

Let $G$ be a compact group and $\widehat{G}$ be the set of equivalence classes of unitary irreducible representations of $G$. The Fourier algebra $A(G)$ of $G$ is defined as
	$$A(G):=\left\{f\in C(G) : \norm{f}_{A(G)}:=\sum_{\pi\in \widehat{G}}d_\pi \norm{\widehat{f}(\pi)}_1 < \infty \right\}.$$
Here, $\widehat{f}(\pi)$ denotes the Fourier transform given by
	$$\widehat{f}(\pi) := \int_G f(x)\ \overline{\pi}(x)\ dx \in M_{d_\pi}(\mathbb{C})$$
where $dx$ denotes the normalized Haar measure on $G$; $\overline{\pi}$ denotes the conjugate representation of $\pi$; and $\norm{\cdot}_1$ is the trace norm.
It is well known that the Fourier algebra is actually a Banach algebra under the pointwise multiplication.

The Fourier algebra can be defined for any locally compact groups (see \cite{Em64}) and is regarded 
as one of the most fundamental examples of commutative Banach algebras associated to groups. 
When the (compact) group $G$ is abelian, $A(G)$ is nothing but the group algebra $L^1(\widehat{G})$ 
of the Pontryagin dual $\widehat{G}$, so that Fourier algebras are usually called the ``dual" object of group 
algebras. In general, Fourier algebras are quite far away from operator algebras (i.e.~norm-closed 
subalgebras of $B(\mathcal{H})$ for some Hilbert space $\mathcal{H}$) including $C^*$-algebras. However, by putting some
weights on $A(G)$ for a compact group $G$ we can make weighted versions of $A(G)$ much closer to 
operator algebras.

We call a function $\om : \widehat{G} \to [1,\infty)$ a {\it weight} if
	\begin{equation}\label{eq-weight1}
	\om(\sigma) \le \om(\pi)\ \om(\pi')
	\end{equation}
for any $\pi,\pi'\in \widehat{G}$ and $\sigma\in \widehat{G}$ appearing as a component of the irreducible decomposition of $\pi \otimes \pi'$.

We define the {\it Beurling-Fourier algebra} $A(G,\om)$ by
	$$A(G,\om):=\left\{f\in C(G) : \norm{f}_{A(G,\om)}=\sum_{\pi\in \widehat{G}}d_\pi\ \om(\pi)\norm{\widehat{f}(\pi)}_1 <\infty \right\}.$$	
There is a natural isometry between $A(G)$ and $A(G,\om)$ (see \cite{LS12} for the details),
so that we can endow an operator space 
structure on $A(G,\om)$ coming from $A(G)$ (as the predual of the group von Neumann algebra $\operatorname{VN}(G)$) 
through this isometry. 
Then from the condition \eqref{eq-weight1} one can show that $A(G,\om)$ 
is a completely contractive Banach algebra under the pointwise multiplication (\cite{LS12}).

Fundamental examples of weights on $\widehat{G}$ are given by the following polynomial dependence on dimensions of the representations. For $\alpha \geq 0$, we define $\om_\alpha : \widehat{G} \to [1,\infty)$ by
	$$\om_\alpha(\pi) = d^\alpha_\pi \;\; \ \ (\pi\in \widehat{G}).$$
Clearly $\om_\alpha$ satisfies the condition \eqref{eq-weight1}, and so, it defines a weight on $\widehat{G}$; it is called the {\it dimension weight of order $\alpha$}.

In \cite[Theorem 4.9]{GhandehariLeeSameiNico2012} it has been shown that $A(\SU(n), \om_\alpha)$ is completely 
isomorphic to an operator algebra under assumption that the 
estimate \eqref{eq:main-sln} for $\SU(n)$ holds true (this assumption was referred to as \cite[Conjecture 1]{GhandehariLeeSameiNico2012}). 
Since our main result says that the 
conjecture is indeed true for all $n\ge 2$, this implies the following.

	\begin{theorem}\label{thm-SU(n)-dim}
	Let $\om_\alpha$ be the dimension weight of order $\alpha>\frac{d(\SU(n))}{2}=\frac{n^2-1}{2}$ on $\widehat{\SU(n)}$, $n\ge 2$. Then $A(\SU(n), \om_\alpha)$ is completely isomorphic to an operator algebra.
	\end{theorem}

Note that the above result is not true for $U(n)$, $n\ge 2$ (in general, for any compact connected non-simple Lie groups, see \cite[Theorem 4.7]{GhandehariLeeSameiNico2012}) even though the representations of $U(n)$, $n\ge 2$ satisfy the estimate \eqref{eq:main-gln}.

\subsection{Viewpoint of representation theory and random matrix theory}

The main result of this paper is also of intrinsic interest in representation theory and also random matrix theory. 
According to it, the `widths' of representations tell something about the relative dimensions
of the Littlewood-Richardson components, namely any irreducible representation appearing 
in the tensor product cannot have a too large relative dimension if the width of both tensored irreducible representations is large enough. 
This result was known for `typical' irreducible representations 
(see e.g. \cite{CS09})
but here we show that it holds true uniformly, at the expense of a worse, but  
asymptotically optimal estimate.
Thus the difficulty of our main result lies in its uniformity.

Our estimate relies on a combinatorial lemma proved in Section \ref{sec:firstrowbound} and it
turns out that this lemma admits a direct counterpart in random matrix that has its own interest.
We state this as Lemma \ref{cor4rmt}.

\subsection{Organization of the paper}
In Section \ref{sec:representation-theory} we recall some notations and facts
from representation theory.
In Section \ref{subsec:LR-measure} we give an auxiliary result: 
a convenient description of the probability distribution of the first coordinate $\mu_1$ of a random representation $\mu=(\mu_1,\mu_2,\dots)$ distributed according to the Littlewood-Richardson measure $P_{\lambda,\mu}$.
Following this description, Section \ref{sec:firstrowbound} gathers the properties of this probability distribution which are necessary in order to prove our main theorem.
Section \ref{sec:proof-main-result} contains the proof of the main theorem, 
and in Section \ref{sec:optimal} we explain the sense in which our result is optimal.

\section{Representation theory of classical groups}
\label{sec:representation-theory}

\subsection{Representations of linear groups $\GL(n)$ and weights}
In this article  $n\geq 1$ is a fixed integer. We say that $\lambda$ is a \emph{weight} if 
$\lambda=(\lambda_1,\dots,\lambda_n)\in\Z^n$ is such that $\lambda_1\geq \cdots\geq \lambda_n$.
We denote by $\widehat{\GL(n)}$ the collection of irreducible representations of the linear group  $\GL(n)$, up to equivalence.
There is a canonical bijective correspondence between the set $\widehat{\GL(n)}$ of (equivalence classes of) irreducible representations and the set of weights
which to a representation associates its \emph{highest weight}.
In order to simplify the notation we will identify an irreducible representation of $\GL(n)$ with the corresponding weight.
We refer to \cite{Fulton} for an extensive treatment of the subject. 
Throughout the whole paper, we work with the field of complex numbers $\C$. In particular, $\GL(n)$ means the linear group $\GL(n,\C)$ and $\SL(n)$ means the special linear group $\SL(n,\C)$.

\subsection{Kronecker tensor product}

If $\rho_1:\GL(n)\rightarrow \End V_1$ and $\rho_2:\GL(n)\rightarrow \End V_2$ are representations 
of the same group $\GL(n)$, we denote by $\rho_1\otimes \rho_2:\GL(n)\rightarrow\End (V_1\otimes V_2)$ 
their \emph{Kronecker tensor product} given by the diagonal action on simple tensors:
$$ \big((\rho_1\otimes \rho_2)(g)\big) (v\otimes w) := \rho_1(g)(v) \otimes \rho_2(g)(w)
$$ 
for $g\in\GL(n)$, $v\in V_1$, $w\in V_2$.

\subsection{Representations of $\SL(n)$}
\label{subsec:sln}
Here we describe briefly the irreducible representations of the special linear group $\SL(n)$ 
of matrices of determinant one,
and their relation 
with the irreducible representations of $\GL(n)$.
It is known, cf \cite[Section 15.5]{MR93a:20069},
that any irreducible representation of $\GL(n)$, when restricted to $\SL(n)$, yields again an irreducible representation. 
Besides, this map is surjective and its quotient can be precisely
described as follows: two representations $\lambda$, $\mu$ of $\GL(n)$ yield the same representation 
when restricted to $\SL(n)$ if and only if there exists an integer $k$ such that $\mu +k \mathbf{1}=\lambda$.

Unsurprisingly, 
the one-dimensional representation given by the determinant is trivial on $\SL(n)$ but non-trivial on $\GL(n)$. Its highest weight is equal to $\mathbf{1}=(1,\ldots , 1)$. 
The highest weight of the trivial representation is equal to $(0,\ldots ,0)$.
As we have seen, they restrict to the same representation of $\SL(n)$.

Put differently, it is possible to parametrize the irreducible representations of $\SL(n)$ as those weights $\lambda=(\lambda_1,\dots,\lambda_n)$ for which the last component is equal to zero: $\lambda_n=0$.

We are now ready to show Corollary \ref{cor:main-theorem-SLn}, assuming that Theorem \ref{theorem:main-theorem} holds true.
\begin{proof}[Proof that Theorem \ref{theorem:main-theorem} implies Corollary \ref{cor:main-theorem-SLn}]
Let $\lambda,\mu$ be (as in Corollary \ref{cor:main-theorem-SLn}) representations of $\SL(n)$.
We view them as weights such that their last components are equal to zero:
$\lambda_n=0$, $\mu_n=0$. These weights give rise to representations
of $\GL(n)$ which will be denoted by $\widetilde{\lambda}$, $\widetilde{\mu}$.

The tensor product $\lambda\otimes \mu$ of representations of $\SL(n)$ is nothing else but a restriction of the tensor product $\widetilde{\lambda}\otimes\widetilde{\mu}$ of representations of $\GL(n)$.
Furthermore, the decomposition of $\widetilde{\lambda}\otimes\widetilde{\mu}$ into irreducible components gives rise (by restriction) to a decomposition of $\lambda\otimes\mu$ into irreducible components.
It follows that
the initial assumption that $[\nu]$ appears in the decomposition of the tensor product $\lambda\otimes\mu$ 
implies that there exists some weight $\nu'$ such that:
\begin{itemize}
 \item $\nu'$ contributes to the decomposition of the tensor product $\widetilde{\lambda}\otimes \widetilde{\mu}$ of representations of $\GL(n)$; 
in other words $c_{\widetilde{\lambda},\widetilde{\mu}}^{\nu'}\geq 1$;
 \item $\nu'$ is an irreducible representation of $\GL(n)$ which restricted to $\SL(n)$ coincides with representation $\nu$.
\end{itemize}

We apply Theorem \ref{theorem:main-theorem} for $\widetilde{\lambda}$, $\widetilde{\mu}$, $\nu'$; Equation \eqref{eq:main-gln} takes the form
$$\frac{c_{\widetilde{\lambda},\widetilde{\mu}}^{\nu'} \ d_\nu}{d_\lambda \ d_\mu}=\frac{c_{\widetilde{\lambda},\widetilde{\mu}}^{\nu'} \ d_{\nu'}}{d_{\widetilde{\lambda}} \ d_{\widetilde{\mu}}}  \leq C_n \left(\frac{1}{\lambda_1}+\frac{1}{\mu_1} \right).$$
Taking into account $c_{\widetilde{\lambda},\widetilde{\mu}}^{\nu'}\geq 1$ this implies Equation \eqref{eq:main-sln} and finishes the proof of Corollary \ref{cor:main-theorem-SLn}.
\end{proof}

\section{Littlewood-Richardson measure and Gelfand-Tsetlin patterns}
\label{subsec:LR-measure}

\newcommand{\ler}{\rotatebox[origin=c]{90}{$\leq$}}

In Section \ref{subsec:gen-tableau} we  recall the definition of \emph{Gelfand-Tsetlin patterns}. 
As we shall see, patterns provide a concrete model for Littlewood-Richardson measure 
(Lemma \ref{lem:first-row} \ref{enum:lem:first-row}). For the purposes of the current paper we do not need this 
kind of result in full generality; for this reason in Section \ref{subsec:keylemma} we will state 
Lemma \ref{lem:first-row-A} which concerns the simplified setup: the first coordinate of a random 
weight distributed according to Littlewood-Richardson measure. This lemma is the key element of 
the proof of Lemma \ref{lem:firstrowbound} in Section \ref{sec:firstrowbound}, which will be used in 
the proof of Theorem \ref{theorem:main-theorem} (the main theorem). The remaining part of this section 
is devoted to the proof of Lemma \ref{lem:first-row-A}.

\subsection{Gelfand-Tsetlin patterns}
\label{subsec:gen-tableau}

Let $\lambda$ be a weight. We say that
$$A=\big(a_l(i)\big)_{l\in \{i,\ldots, n-1\},\ i\in\{1,\ldots, n-1\}}\in \Z^{n(n-1)/2}$$
is a \emph{Gelfand-Tsetlin pattern} of shape $\lambda$ (or, shortly, \emph{pattern}) if the following system of inequalities is fulfilled:
\begin{equation}
\label{eq:system}
\small
\begin{matrix}
a_1(1) & \leq & a_2(1) & \leq & \cdots & \leq & a_{n-2}(1) & \leq & a_{n-1}(1) & \leq & \lambda_1 \\
\ler &  & \ler &  & &   &  \ler & &\ler \\
a_2(2) & \leq & a_3(2) & \leq & \cdots & \leq & a_{n-1}(2) & \leq  & \lambda_2 \\
 \ler & & \ler \\
\vdots & & \vdots \\ 
 \ler & & \ler \\
a_{n-1}(n-1) & \leq & \lambda_{n-1} \\
  \ler \\
\lambda_n.
\end{matrix}
\end{equation}
This system of inequalities can be represented by an oriented graph $\G$ from Figure \ref{fig:graph}. 

The first row of \eqref{eq:system} will deserve special attention, for this reason we will use simplified notation
\begin{align*}
 a_l& :=a_l(1) \qquad \text{for } l\in\{1,\dots,n-1\}. \\
\intertext{It will be also convenient to define}
 a_n & :=\lambda_1.
\end{align*}
Analogously, if $B=\big(b_l(i)\big)$
is a pattern of shape $\mu$ we denote
\begin{align*} 
b_l & :=b_l(1) \qquad \text{for } l\in\{1,\dots,n-1\} \\
\intertext{and}
 b_n& :=\mu_1.
\end{align*}

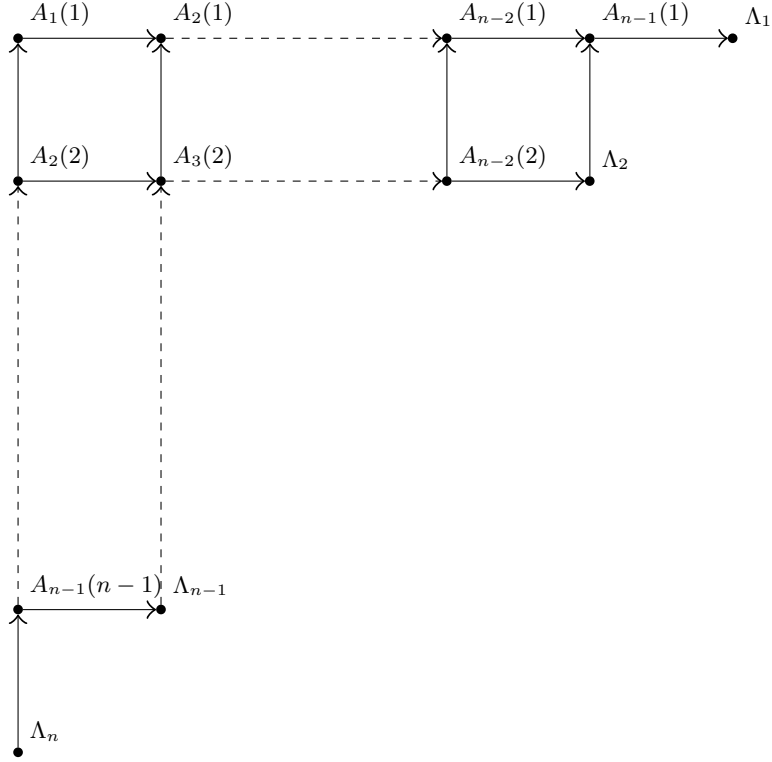
\begin{figure}
\centerline{
\begin{tikzpicture}[scale=1.9,nod/.style={circle,draw=black,fill=black,thick,
inner sep=0pt,minimum size=1mm}]
\scriptsize
\node[label=45:$A_{1}(1)$] (a11) at (0,0) [nod] {};
\node[label=45:$A_{2}(1)$] (a21) at (1,0) [nod] {};
\node[label=45:$A_{2}(2)$] (a12) at (0,-1) [nod] {};
\node[label=45:$A_{3}(2)$] (a22) at (1,-1) [nod] {};
\draw[decoration={markings,mark=at position 1 with {\arrow[scale=2]{>}}},
    postaction={decorate},
    shorten >=0.4pt] (a11) -- (a21) ;
\draw[decoration={markings,mark=at position 1 with {\arrow[scale=2]{>}}},
    postaction={decorate},
    shorten >=0.4pt] (a12) -- (a22) ;
\draw[decoration={markings,mark=at position 1 with {\arrow[scale=2]{>}}},
    postaction={decorate},
    shorten >=0.4pt] (a12) -- (a11) ;
\draw[decoration={markings,mark=at position 1 with {\arrow[scale=2]{>}}},
    postaction={decorate},
    shorten >=0.4pt] (a22) -- (a21) ;
\node[label=45:$A_{n-2}(1)$] (an11) at (3,0) [nod] {};
\node[label=45:$A_{n-2}(2)$] (an12) at (3,-1) [nod] {};
\node[label=45:$A_{n-1}(1)$] (an21) at (4,0) [nod] {};
\node[label=45:$\Lambda_1$]  (an31) at (5,0) [nod] {};
\node[label=45:$\Lambda_2$]  (an22) at (4,-1) [nod] {};
\draw[decoration={markings,mark=at position 1 with {\arrow[scale=2]{>}}},
    postaction={decorate},
    shorten >=0.4pt] (an11) -- (an21) ;
\draw[decoration={markings,mark=at position 1 with {\arrow[scale=2]{>}}},
    postaction={decorate},
    shorten >=0.4pt] (an21) -- (an31) ;
\draw[decoration={markings,mark=at position 1 with {\arrow[scale=2]{>}}},
    postaction={decorate},
    shorten >=0.4pt] (an22) -- (an21) ;
\draw[decoration={markings,mark=at position 1 with {\arrow[scale=2]{>}}},
    postaction={decorate},
    shorten >=0.4pt] (an12) -- (an11) ;
\draw[decoration={markings,mark=at position 1 with {\arrow[scale=2]{>}}},
    postaction={decorate},
    shorten >=0.4pt] (an12) -- (an22) ;
% % % % % % % % % % % % % % % % % % % % % % % % % % % % % 
\draw[dashed,decoration={markings,mark=at position 1 with {\arrow[scale=2]{>}}},
    postaction={decorate},
    shorten >=0.4pt] (a21) -- (an11);
\draw[dashed,decoration={markings,mark=at position 1 with {\arrow[scale=2]{>}}},
    postaction={decorate},
    shorten >=0.4pt] (a22) -- (an12);
% % % % % % % % % % % % % % % % % % % % % % % % % % % % % 
\node[label=45:$A_{n-1}(n-1)$] (am11) at (0,-4) [nod] {};
\node[label=45:$\Lambda_n$] (am12) at (0,-5) [nod] {};
\node[label=45:$\Lambda_{n-1}$] (am21) at (1,-4) [nod] {};
\draw[dashed,decoration={markings,mark=at position 1 with {\arrow[scale=2]{>}}},
    postaction={decorate},
    shorten >=0.4pt] (am11) -- (a12);
\draw[dashed,decoration={markings,mark=at position 1 with {\arrow[scale=2]{>}}},
    postaction={decorate},
    shorten >=0.4pt] (am21) -- (a22);
\draw[decoration={markings,mark=at position 1 with {\arrow[scale=2]{>}}},
    postaction={decorate},
    shorten >=0.4pt] (am11) -- (am21) ;
\draw[decoration={markings,mark=at position 1 with {\arrow[scale=2]{>}}},
    postaction={decorate},
    shorten >=0.4pt] (am12) -- (am11) ;
\end{tikzpicture}
}
\caption{Oriented graph $\G$ corresponding to the system of inequalities \eqref{eq:system}.}
\label{fig:graph}
\end{figure}

\subsection{%The key lemma: 
Concrete realization of Littlewood-Richardson measure}
\label{subsec:keylemma}
The following proposition is the key component in the proof of Proposition \ref{lem:firstrowbound}. 
It gives a concrete realization of the first coordinate of a random weight distributed 
according to Littlewood-Richardson measure.

\begin{proposition}
\label{lem:first-row-A}
Let $\lambda,\mu$ be weights.
Let $A=\big(a_l(i)\big)$ be a random pattern of shape $\lambda$ 
(sampled with the uniform distribution)
and let $B=\big(b_l(i)\big)$ be a random pattern of shape $\mu$ 
(also sampled with the uniform distribution), we assume that $A$ and $B$ are independent. 

Let $\nu=(\nu_1,\dots,\nu_n)$ be a random weight distributed according to the 
Littlewood-Richardson measure  $P_{\lambda,\mu}$; then
\begin{equation}
\label{eq:combinatorial-B}
 \nu_1 \stackrel{\text{dist}}{=} \max_{\substack{k,l\geq 1, \\ k+l=n+1}}  a_k + b_l,
\end{equation}
where $\stackrel{\text{dist}}{=}$ denotes the equality of distributions of random variables.
\end{proposition}
We postpone its proof until Section \ref{sec:proof:em:first-row-A}.
The remaining part of the current section is devoted to preparation to this proof.

\subsection{Polynomial representations}

\emph{Polynomial} irreducible representations of $\GL(n)$ play a special role.
Such a polynomial representation corresponds to a weight $(\lambda_1,\dots,\lambda_n)\in\Z^n$ 
such that $\lambda_1\geq \cdots \geq \lambda_n\geq 0$ are \emph{non-negative} integers. 
A weight with this property is called a \emph{Young diagram} and can be represented 
graphically as shown on Figure \ref{fig:young-diagram} (we use the English notation for 
drawing Young diagrams). Polynomial representations 
are associated to very rich combinatorial structures related to Young diagrams 
and \emph{Young tableaux} which we will explore in Section \ref{subsec:Young-tableaux}.

\begin{figure}
\centering
\begin{tikzpicture}
\draw[thick] (0,0) -- (9,0) -- (9,-1) -- (7,-1) -- (7,-2) -- (3,-2) -- (3,-3) -- (0,-3) -- cycle;
\clip (0,0) -- (9,0) -- (9,-1) -- (7,-1) -- (7,-2) -- (3,-2) -- (3,-3) -- (0,-3);
\draw (0,-3) grid (10,0);
\end{tikzpicture}

\caption{Young diagram $(9,7,3)$.}
\label{fig:young-diagram}
\end{figure}
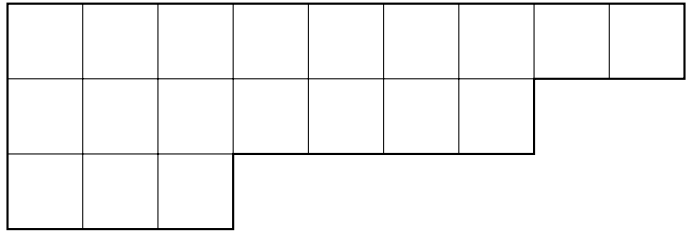

Many problems concerning irreducible representations can be reduced to the special case 
of irreducible \emph{polynomial} representations. This is also the case for Lemma 
\ref{lem:first-row-A}, the following lemma gives the details of this reduction.
\begin{lemma}
\label{lem:reduction-to-tableaux}
Assume that Lemma \ref{lem:first-row-A} is true under the additional assumption that 
weights $\lambda,\mu$ are \emph{Young diagrams}. Then Lemma \ref{lem:first-row-A} is true 
in general, without such an assumption.
\end{lemma}
\begin{proof}
\newcommand{\ekspoA}{p}
\newcommand{\ekspoB}{q}
For $\ekspoA\in\Z$ we denote by $\Det^\ekspoA:\GL(n)\rightarrow \End(\C)$
the one-dimensional representation given by 
an appropriate power of the determinant:
$$ \Det^\ekspoA(g):= \big( \det(g) \big)^\ekspoA \qquad \text{for $g\in\GL(n)$}, $$
where the right-hand side should be interpreted as a $1\times 1$ matrix, thus as an endomorphism 
of the one-dimensional vector space $\C=\C^1$.
Representation $\Det^\ekspoA$ is irreducible and corresponds to the highest weight 
$$\ekspoA\one:=(\ekspoA,\dots,\ekspoA)\in\Z^n.$$

Kronecker tensor product $\lambda\otimes \Det^\ekspoA$ of an irreducible representation 
$\lambda=(\lambda_1,\dots,\lambda_n)$ with $\Det^\ekspoA$ is again an irreducible representation 
which corresponds to the shifted weight 
$$\lambda+\ekspoA\one:=(\lambda_1+\ekspoA,\dots,\lambda_n+\ekspoA).$$

The dimensions of irreducible representations, Littlewood-Richardson coefficients 
and the Littlewood-Richardson measure are invariant under 
such shifts: 
\begin{align*} d_{\lambda+\ekspoA \one} =&d_\lambda, \\
 c_{\lambda+\ekspoA \one, \mu+\ekspoB \one}^{\nu+(\ekspoA+\ekspoB) \one} =& c_{\lambda,\mu}^\nu, \\
 P_{\lambda+\ekspoA \one, \mu+\ekspoB \one}\big(\nu+(\ekspoA+\ekspoB) \one\big) =& P_{\lambda,\mu}(\nu),
\end{align*}
for arbitrary $\ekspoA,\ekspoB\in \Z$ and irreducible representations $\lambda,\mu,\nu$ of $\GL(n)$.

We use notations of Lemma \ref{lem:first-row-A}.
We denote $\lambda=(\lambda_1,\dots,\lambda_n)$, $\mu=(\mu_1,\dots,\mu_n)$ and set 
$\ekspoA:=-\lambda_n$ and $\ekspoB:=-\mu_n$ so that weights $\lambda':=\lambda+\ekspoA \one$ 
and $\mu':=\mu+\ekspoB \one$ are \emph{Young diagrams}. We also set 
$\nu'=(\nu_1',\dots,\nu_n'):=(\ekspoA+\ekspoB) \one+\nu$.
Clearly, since $\nu$ is distributed according to Littlewood-Richardson measure 
$P_{\lambda,\mu}$ it follows that $\nu'$ is distributed according to Littlewood-\linebreak{}Richardson measure 
$P_{\lambda',\mu'}$.

We define shifted patterns $A'=\big(a_l(i)+\ekspoA \big)$ and 
$B'=\big(b_l(i)+\ekspoB \big)$. Clearly $A'$ and $B'$ are random patterns of 
shape $\lambda'$ and $\mu'$ respectively.

We apply Lemma \ref{lem:first-row-A} to Young diagrams $\lambda'$, $\mu'$, random weight $\nu'$ 
and random  patterns $A'$, $B'$. It follows that
$$\nu_1+(\ekspoA+\ekspoB)=\nu'_1 \stackrel{\text{dist}}{=} \max_{\substack{k,l\geq 1, \\ k+l=n+1}}  
a'_k + b'_l= \max_{\substack{k,l\geq 1, \\ k+l=n+1}}  (a_k+\ekspoA) + (b_l+\ekspoB)$$
which shows that Lemma \ref{lem:first-row-A} holds true for weights $\lambda$ and $\mu$ as desired.
\end{proof}

\subsection{Young tableaux, Robinson-Schensted-Knuth correspondence and the plactic monoid}
\label{subsec:Young-tableaux}
We recall some basic notations related to Young tableaux, Robinson-Schensted-Knuth 
correspondence and the plactic monoid. A good treatment of these topics is given in Part I of the book \cite{Fulton}.

\subsubsection{Tableaux}
\label{subsec:tableaux}

\newcommand{\numA}{\textcolor{red}{1}}
\newcommand{\numB}{\textcolor{PineGreen}{2}}
\newcommand{\numC}{\textcolor{blue}{3}}

\begin{figure}
\centering
\begin{tikzpicture}
\fill[pattern color=blue!20,pattern=north west lines] (0,0) -- (9,0) -- (9,-1) -- (7,-1) -- (7,-2) -- (3,-2) -- (3,-3) -- (0,-3) -- cycle;
\fill[fill=green!10] (0,0) -- (8,0) -- (8,-1) -- (4,-1) -- (4,-2) -- (0,-2) -- cycle;
\fill[fill=white] (0,0) -- (5,0) -- (5,-1)  -- (0,-1) -- cycle;
\fill[pattern color=red!20,pattern=north east lines] (0,0) -- (5,0) -- (5,-1)  -- (0,-1) -- cycle;
\draw[thick] (0,0) -- (9,0) -- (9,-1) -- (7,-1) -- (7,-2) -- (3,-2) -- (3,-3) -- (0,-3) -- cycle;
\clip (0,0) -- (9,0) -- (9,-1) -- (7,-1) -- (7,-2) -- (3,-2) -- (3,-3) -- (0,-3);
\draw (0,-3) grid (10,0);
\node at (0.5,-0.5) {$1$};
\node at (1.5,-0.5) {$1$};
\node at (2.5,-0.5) {$1$};
\node at (3.5,-0.5) {$1$};
\node at (4.5,-0.5) {$1$};
\node at (5.5,-0.5) {$2$};
\node at (6.5,-0.5) {$2$};
\node at (7.5,-0.5) {$2$};
\node at (8.5,-0.5) {$3$};
% % % % % % % % % % % % % % % % 
\node at (0.5,-1.5) {$2$};
\node at (1.5,-1.5) {$2$};
\node at (2.5,-1.5) {$2$};
\node at (3.5,-1.5) {$2$};
\node at (4.5,-1.5) {$3$};
\node at (5.5,-1.5) {$3$};
\node at (6.5,-1.5) {$3$};
% % % % % % % % % % % % % % % % 
\node at (0.5,-2.5) {$3$};
\node at (1.5,-2.5) {$3$};
\node at (2.5,-2.5) {$3$};
\end{tikzpicture}

\caption{Example of a tableau $T$ in the alphabet $\{1,2,3\}$ filling the Young diagram $(9,7,3)$ 
from Figure \ref{fig:young-diagram}. The boxes were colored in order to improve visibility.
The corresponding word is given by \protect\linebreak $w(T)=(\numC,\numC,\numC,\numB,\numB,\numB,\numB,\numC,\numC,\numC,\numA,\numA,\numA,\numA,\numA,\numB,\numB,\numB,\numC)$.}
\label{fig:tableau}
\end{figure}

A \emph{semi-standard tableau} (or, shortly, \emph{tableau}) $T$ is a filling of the boxes a given 
Young diagram $\lambda$ with letters from the alphabet $\{1,\dots,n\}$ with the property that the filling 
should be weakly increasing along
each row, and strictly increasing down a column, see Figure \ref{fig:tableau}. 
The value of $n$ will be fixed so we do not have to specify it for each tableau separately.
We also say that \emph{Young diagram $\lambda$ is the shape of tableau $T$}.

For a given tableau $T$ we set $a_l(i)$ to be the number of boxes in the $i$th row of $T$ filled with numbers 
$\leq l$. It is easy to check that so defined $A=\big(a_l(i)\big)$ is a pattern; furthermore for any 
Young diagram $\lambda$ this gives a bijective correspondence between tableaux of shape 
$\lambda$ and patterns of shape $\lambda$. In the following we will identify a tableau 
with the corresponding pattern.

\subsubsection{Words}
A \emph{word} $w=(w_1,\dots,w_\ell)$ is a sequence of the elements of the alphabet $\{1,\dots,n\}$. 
We recall that the \emph{insertion tableau} $P(w)$ of $w$ is defined as the semi-standard tableau 
obtained by \emph{Schensted row insertion algorithm} 
applied iteratively to the letters $w_1,\dots,w_\ell$. For a given tableau $T$ we denote by $w(T)$ the 
word obtained by reading the entries of $T$ along the lines, from left to right and from the bottom line 
to the top one, see Figure \ref{fig:tableau}. This word has a property that $T=P\big(w(T)\big)$.

For a word $w=(w_1,\dots,w_\ell)$ we denote by $\LIS(w)$ the
\emph{length of the longest (weakly) increasing subsequence} of $w$, i.e.~the length of the longest 
sequence $i_1<\cdots<i_k\in\{1,\dots,\ell\}$ such that
$$ w_{i_1} \leq \cdots \leq w_{i_k}.$$
It is well-known 
that if $\lambda=(\lambda_1,\dots,\lambda_n)$ is the shape of the 
insertion tableau $P(w)$ then $\LIS(w)=\lambda_1$ is equal to the length of the first row of $\lambda$.

\subsubsection{Multiplication of tableaux, plactic monoid and plactic Littlewood-Richardson rule}

We consider the \emph{free monoid} in alphabet $\{1,\dots,n\}$, which is just the set of words equipped 
with a multiplication $\cdot$ given by concatenation of words. Let us identify two words $w$ and 
$w'$ (we denote it $w\equiv w'$) if and only if the corresponding insertion tableaux are equal: $P(w)=P(w')$. 
One can show that $w\equiv w'$ and $v\equiv v'$ implies that $w\cdot v\equiv w'\cdot v'$ thus 
multiplication $\cdot$ is well defined on the equivalence classes of $\equiv$.
The set of such equivalence classes of $\equiv$ equipped with multiplication $\cdot$ is 
called \emph{plactic monoid}. 

Map $P$ gives a bijection between the elements of the plactic monoid and tableaux; thus the multiplication 
in the plactic monoid can be used to define \emph{multiplication of tableaux} which will be denoted by 
the same symbol $\cdot$. Alternatively, the product $S \cdot T:=P\big( w(S) \cdot w(T) \big)$ of tableaux 
$S$ and $T$ is defined as the insertion tableau corresponding to the concatenation of the words 
corresponding to the original tableaux.

% \subsubsection{Plactic Littlewood-Richardson rule}

Recall that the \emph{plactic Schur polynomial} 
is defined as a formal sum
$$ S_\lambda := \sum_T T$$
of all  tableaux with shape $\lambda$. 
\emph{Plactic Littlewood-Richardson rule} 
says that
\begin{equation}
\label{eq:plactic-LR}
 S_\lambda \cdot S_\mu = \sum_\nu c_{\lambda,\mu}^\nu S_\nu, 
\end{equation}
where $c_{\lambda,\mu}^\nu$ are the usual Littlewood-Richardson coefficients.

\subsubsection{Involution on tableaux}
Let us consider an antiautomorphism $\alpha$ of the free monoid defined on the generators by 
$\alpha(i):=n+1-i$. Alternatively, $\alpha$ is an involution on words defined by reading the word 
backwards and by reversing the order in the alphabet. Plactic monoid can be equivalently described 
as the free monoid divided by \emph{plactic relations} (\emph{Knuth relations})
which are fulfilled by generators $x,y,z\in\{1,\dots,n\}$:
\begin{align*} 
y\cdot z\cdot x & = y\cdot x \cdot z \qquad\text{if } x < y \leq z, \\
x\cdot z\cdot y & = z\cdot x \cdot y \qquad\text{if } x \leq y < z.
\end{align*}
Since $\alpha$ preserves these plactic relations, $\alpha$ gives rise to an antiautomorphism of the plactic monoid. 

If we identify the elements of the plactic monoid with tableaux, the antiautomorphism $\alpha$ 
becomes an involution on the set of tableaux. It can be described explicitly as follows:
for a given tableau $T$ we replace each entry $i$ by $\alpha(i)=n+1-i$ and we rotate the tableau by angle $\pi$, 
thus obtaining a \emph{skew tableau}, see Figure \ref{fig:tableau-rotated}. After rectifying it 
(by an application of Sch\"utzerberger's \emph{jeu de taquin}),
we obtain $\alpha(T)$. Alternatively, $\alpha(T)= P\big( \alpha(w(T) \big)$.
Greene's theorem 
shows that involution $\alpha$ maps the set of tableaux of a 
given shape into itself.

\begin{figure}
\centering
\begin{tikzpicture}
\fill[pattern color=blue!20,pattern=north west lines] (-0,-0) -- (-9,-0) -- (-9,1) -- (-7,1) -- (-7,2) -- (-3,2) -- (-3,3) -- (0,3) -- cycle;
\fill[fill=green!10] (0,0) -- (-8,0) -- (-8,1) -- (-4,1) -- (-4,2) -- (0,2) -- cycle;
\fill[fill=white] (0,0) -- (-5,0) -- (-5,1)  -- (0,1) -- cycle;
\fill[pattern color=red!20,pattern=north east lines] (0,0) -- (-5,0) -- (-5,1)  -- (0,1) -- cycle;
\draw[thick] (0,0) -- (-9,0) -- (-9,1) -- (-7,1) -- (-7,2) -- (-3,2) -- (-3,3) -- (0,3) -- cycle;
\clip (-0,-0) -- (-9,-0) -- (-9,1) -- (-7,1) -- (-7,2) -- (-3,2) -- (-3,3) -- (0,3) ;
\draw (-10,0) grid (0,3);
\node at (-0.5,0.5) {$3$};
\node at (-1.5,0.5) {$3$};
\node at (-2.5,0.5) {$3$};
\node at (-3.5,0.5) {$3$};
\node at (-4.5,0.5) {$3$};
\node at (-5.5,0.5) {$2$};
\node at (-6.5,0.5) {$2$};
\node at (-7.5,0.5) {$2$};
\node at (-8.5,0.5) {$1$};
% \node at (-9.5,0.5) {$1$};
% % % % % % % % % % % % % % % % 
\node at (-0.5,1.5) {$2$};
\node at (-1.5,1.5) {$2$};
\node at (-2.5,1.5) {$2$};
\node at (-3.5,1.5) {$2$};
\node at (-4.5,1.5) {$1$};
\node at (-5.5,1.5) {$1$};
\node at (-6.5,1.5) {$1$};
% % % % % % % % % % % % % % % % 
\node at (-0.5,2.5) {$1$};
\node at (-1.5,2.5) {$1$};
\node at (-2.5,2.5) {$1$};
\end{tikzpicture}

\caption{Skew tableau obtained from the tableau from Figure \ref{fig:tableau} after rotating by angle $\pi$ and
replacing each entry $i$ by $\alpha(i)=n+1-i$.}
\label{fig:tableau-rotated}
\end{figure}
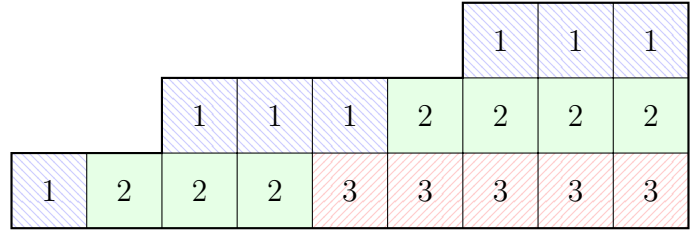

\subsection{Concrete model for Littlewood-Richardson measure}

The following lemma is a simple reformulation of well-known combinatorics of the 
representation theory in the language of probability theory.

The following is the key ingredient for the proof of Proposition \ref{lem:first-row-A}.
\begin{lemma}
\label{lem:first-row}
Let $\lambda,\mu$ be Young diagrams.
Let $S$ be a random  Young tableau of shape $\lambda$  
and let $T$ be a random  Young tableau of shape $\mu$.
We assume that $S$ and $T$ are sampled according to the uniform distribution given by their respective shape
constraints, and that they are independent.

Then,
\begin{enumerate}[label=\emph{(\alph*)}]
 \item \label{enum:lem:first-row} the distribution of the shape of the product $S \cdot T$ coincides with the 
 Littlewood-Richardson measure $P_{\lambda,\mu}$;
 \item \label{enum:lem:first-rowXXX} let $\nu=(\nu_1,\dots,\nu_n)$ be a random Young diagram distributed 
 according to the Littlewood-Richardson measure  $P_{\lambda,\mu}$, then
$$ \nu_1 \stackrel{\text{dist}}{=} \max_{\substack{k,l\geq 1, \\k+l=n+1}}  a_k(S) + a_l(T),$$
where $\stackrel{\text{dist}}{=}$ denotes the equality of distributions of random variables.
\end{enumerate}
\end{lemma}
\begin{proof}

We will identify a probability measure on the set of tableaux with the appropriate formal linear 
combination of tableaux with coefficients given by appropriate probabilities.
The dimension $d_\lambda$ is equal to the number of  tableaux of the shape 
$\lambda$, 
therefore the normalized plactic Schur polynomial
$$ \frac{1}{d_\lambda} S_\lambda $$
can be identified
to the uniform probability measure on the set of tableaux of shape $\lambda$. 

The plactic Littlewood-Richardson rule \eqref{eq:plactic-LR} can be equivalently written in the form
$$ \left( \frac{1}{d_\lambda}  S_\lambda \right) \cdot
\left( \frac{1}{d_\mu} S_\mu \right)= \sum_\nu \left( \frac{d_\nu \ c_{\lambda,\mu}^\nu}{d_\lambda d_\mu} \right)\left( \frac{1}{d_\nu} S_\nu\right).$$
The left-hand side corresponds to the distribution of the random tableau \mbox{$S \cdot T$}.
The right-hand side corresponds to the distribution of the random tableau filling a random Young diagram 
with the distribution $P_{\lambda,\mu}$. By comparing the distribution of the shape of the Young 
tableaux contributing to both sides of the equality we finish the proof of the first part of the lemma.

Let $\nu=(\nu_1,\dots,\nu_n)$ be the shape of the tableau $S\cdot T$; from the first part of this lemma it follows 
that the distribution of $\nu$ is given by the Littlewood-Richardson measure $P_{\lambda,\mu}$. Clearly,  
the length of the first row of $\nu$ fulfills
$$\nu_1=\LIS\big( w(S)\cdot w(T)\big);$$ 
it follows that
\begin{equation}
\label{eq:first-row}
\nu_1 = \max_k   \left[ \LIS \left( w(S)\big|_{\{1,\dots,k\}} \right) + \LIS \left( w(T)\big|_{\{k,\dots,n\}}\right) \right],
\end{equation}
where $w|_A$ denotes the word $w$ with all letters which do not belong to $A$ omitted.
In the following we will analyze the two summands contributing to the right-hand side of \eqref{eq:first-row}. 
We start with the first one.

We consider the tableau $S\big|_{\{1,\dots,k\}}$ obtained by removing from $S$ all boxes with entries bigger 
than $k$.
Clearly, 
$$w(S)\big|_{\{1,\dots,k\}}=w\left(S\big|_{\{1,\dots,k\}}\right).$$ 
In particular, 
\begin{equation}
\label{eq:summandA}
\LIS \left( w(S)\big|_{\{1,\dots,k\}} \right)=a_k(S) 
\end{equation}
is the length of the first row of tableau $S\big|_{\{1,\dots,k\}}$.

We turn now to the second summand on the right-hand side of \eqref{eq:first-row}.
Clearly, for any word $w$
$$ w\big|_{\{k,\dots,n\}} = \alpha\left( \alpha(w) \big|_{\{1,\dots,n+1-k\}} \right)
 $$
and 
$$ \LIS w = \LIS \alpha(w)$$
thus
$$ \LIS \left( w\big|_{\{k,\dots,n\}} \right)= \LIS\left( \alpha(w) \big|_{\{1,\dots,n+1-k\}} \right).
 $$

We define $T'=\alpha(T)$;  thus random tableaux $T'$ and $T$ have the same distribution. We have
\begin{equation}
\label{eq:summandB}
 \LIS \left( w(T)\big|_{\{k,\dots,n\}} \right)=\LIS \left(  w(T')  \big|_{\{1,\dots,n+1-k\}} \right)= a_{n+1-k}(T').  
\end{equation}

Equations \eqref{eq:first-row}, \eqref{eq:summandA}, \eqref{eq:summandB} finish the proof.
\end{proof}

\subsection{Proof of Proposition \ref{lem:first-row-A}}
\label{sec:proof:em:first-row-A}

\begin{proof}[Proof of Proposition \ref{lem:first-row-A}]
In Lemma \ref{lem:reduction-to-tableaux} we showed that it is enough to prove the result under 
additional assumption that $\lambda$ and $\mu$ are Young diagrams.
We use part \ref{enum:lem:first-rowXXX} of Lemma \ref{lem:first-row} and use the fact that there is a 
bijective correspondence between tableaux and patterns.
\end{proof}

\subsection{An application to random matrix theory}

In what follows, we state an interesting corollary of Proposition \ref{lem:first-row-A}.
This corollary is of purely random matrix nature, but to the best of our knowledge it seems to be new.

\begin{corollary}
\label{cor4rmt}
Let $A,B$ be independent Hermitian random matrices of the same size $n\times n$. Assume that both the distribution of $A$ and the distribution of $B$ is
invariant under unitary conjugation.
Then the largest eigenvalue of $A+B$ is a random variable which has the same distribution as
$$\max_{\substack{k,l\geq 1, \\k+l=n+1}} a_k+b_l,$$
where $a_k$ (resp.~$b_k$) is the random variable obtained by taking 
the largest eigenvalue of the $k\times k$ upper left corner of $A$ (resp.~$B$).
\end{corollary}

We will just sketch the main ideas of the proof and leave the details to the reader.

\begin{proof}[Sketch of the proof]
% Without loss of generality we can assume that both $A$ are $B$ have a distribution that is unitarily invariant.
% Indeed, if $B$ is not unitarily invariant, one can show that
% replacing $(A,B)$ by $(VAV^*,VBV^*)$ where $V$ is a random unitary
% invariant from $A,B$ does not affect the distribution of the quantities involved in the equality to be proved.
% % % % % THIS IS NOT TRUE!!!!
Without loss of generality we can assume that the eigenvalues of $A,B$ are prescribed. Indeed, if they are 
random, the proof can be completed by conditioning over prescribed
eigenvalues and a decomposition of measure type argument.

And if the eigenvalues of $A,B$ are prescribed,  the result follows from
Proposition \ref{lem:first-row-A} and successive applications of \cite{CS09}. Indeed, in 
\cite{CS09}, it is proved that if $A$ is a unitarily invariant selfadjoint random 
matrix and $\lambda^N=(\lambda_1^N\geq \cdots \geq \lambda_n^N)$, 
is a tuple of sequences of integers such that $\lambda_i^N/N$ 
converges to the $i$th largest eigenvalue
of $A$, then  the law of $(a_1,\dots,a_n)$ is the limit of the laws of 
$(a_1^N/N,\dots,a_n^N/N)$  as appearing in Proposition \ref{lem:first-row-A}
and corresponding to weight $\lambda^N$.
A similar statement holds for a random matrix $B$ and $\mu^N=(\mu_1^N\geq \cdots \geq \mu_n^N)$. It has been also shown in \cite{CS09} that the law of the largest eigenvalue of $A+B$ is the limit of the laws of $\nu_1^N/N$, where $\nu^N$ is distributed according to the Littlewood-Richardson measure $P_{\lambda^N, \nu^N}$.
We apply Proposition \ref{lem:first-row-A} to $\lambda^N$, $\mu^N$ and $\nu^N$
and pass to the limit.
\end{proof}

\section{The first row of a random pattern}
\label{sec:firstrowbound}

The main result of this section is the following lemma giving an upper bound on the atoms of the distribution 
of the first row of a random pattern with a given shape. This proposition is the key in the 
proof of Theorem \ref{theorem:main-theorem}.

\begin{proposition}
\label{lem:firstrowbound}
There exists some constant $D_n$ with the following property.
Let $\lambda$ be a weight and let $A=\big(a_l(i)\big)$ be a random 
pattern with shape $\lambda$. Then for any $x\in\Z$ and $1\le k \le n-1$:
$$ P\big( a_k = x ) \leq  D_n\  \frac{1}{\lambda_1-\lambda_{n+1-k}}.$$
\end{proposition}
We postpone the proof 
of Proposition \ref{lem:firstrowbound}
until Section \ref{subsec:proof-of-firstrow}; the remaining part of the current section is a preparation for this proof.

\subsection{Taking degeneracy into account}
Let the weight $\lambda$ be fixed. The inequalities \eqref{eq:system} 
define a convex polyhedron in the space $\R^{n(n-1)/2}$.
For some choices of the weight $\lambda$ it might happen that that this polyhedron is of 
dimension smaller then the maximal dimension $\frac{n(n-1)}{2}$. 
This creates some difficulties; in the following, we explain how to avoid them.

Restricting the system of inequalities \eqref{eq:system} to one row and one column implies that
$$ \begin{matrix}
a_{l}(i) & \leq & \cdots & \leq & \lambda_i \\
\ler \\
\vdots \\
\ler \\
\lambda_{n+i-l},
\end{matrix}
$$
in other words if  $\lambda_{n+i-l} = \lambda_i$ then automatically $a_{l}(i)=\lambda_i$. Such variables are trivial from our viewpoint, 
thus it is enough to restrict our attention to the index set
$$ \indeks=\big\{ (l,i): l\in \{i,\ldots, n-1\},\ 
i\in\{1,\ldots, n-1\},\ \lambda_{n+i-l} < \lambda_i \big\}$$
and to study only  variables $\big(a_{l}(i): (l,i)\in \indeks \big)$.
We define $d=|\indeks|$. 
The set of solutions to the above system of inequalities \eqref{eq:system} in integer numbers 
(respectively, real numbers) will be denoted by $\Disc\subset \Z^d$  (respectively, by $\Cont\subset \R^d$).  
Thus there is a natural bijective correspondence between patterns of shape 
$\lambda$ and the elements of $\Disc$.

We denote by $\GG$ the oriented graph $\G$ in which:
\begin{itemize}
 \item every vertex $A_l(i)$ with $(l,i)\notin \indeks$ is glued to the vertex $\Lambda_i$, 
 \item all pairs of vertices $\Lambda_i$ and $\Lambda_j$ are glued together if $\lambda_i=\lambda_j$.
\end{itemize}
The graph $\GG$ encodes all inequalities fulfilled by the variables $\big( a_l(i): (l,i)\in \indeks \big)$.
The following lemma is elementary.
\begin{lemma}
The graph $\GG$ is acyclic.
\end{lemma}

\subsection{Continuous versus discrete}
Our goal is to understand the uniform measure on $\Disc$. There is also a simpler object: the uniform 
measure on $\Cont$. In the following we investigate how these two measure are related to each other. 
The following Lemma addresses the question of 
how intersections of $(b+I)$ with $\Disc$ and $\Cont$ are related to each other,
where the unit cube $I$ is defined as
$$ I = \left\{ \big(a_{l}(i)\big): |a_l(i)| < \frac{1}{2} \right\} \subset \R^d. $$
\begin{lemma}
\label{lem:disc-vs-cont}
There is some constant $C>0$ (which depends only on $n$) with the property that for any weight $\lambda$ 
and any lattice point $b\in \Z^d$
\begin{multline*} b\in \Disc \iff
(b+I) \cap \Disc \neq \emptyset \iff \vol \big[ (b+I) \cap \Cont\big] \geq C \\ \iff 
 (b+I) \cap \Cont  \neq \emptyset.
\end{multline*}
\begin{proof}
Since the lattice point $b$ is the only element of $(b+I) \cap \Z^d$, 
if $(b+I) \cap \Disc$ is non-empty then it is equal to $\{b\}$. This explains why the first two conditions are equivalent.

Now we suppose that $b\in\Disc$. For $m\in\Z$ we denote
$$ \indeks_m = \big\{ (l,i) \in \indeks: b_{l}(i)=m \big\}$$
and we denote by $\Cont_m \subset \R^{|\indeks_m|}$ the set of solutions
of the system of inequalities \eqref{eq:system} over variables $a_{l}(i)$ such that $(l,i)\in \indeks_m$, 
subject to the additional requirement that
$$| a_{l}(i) - m | < \frac{1}{2}.$$
Since 
$ (b+I) \cap \Cont = \prod_m \Cont_m$
(where, in the right hand side of this equality, with the obvious identification of the coordinates, 
the multiplication denotes the Cartesian product),
it is enough to show that if $\indeks_m\neq \emptyset$, then
$ \vol \Cont_m$
is bigger than some universal positive constant.

We denote by $\GG_m$ the graph $\GG$ restricted to the following vertices:
\begin{itemize}
 \item vertices $A_l(i)$ with $(l,i)\in \indeks_m$,
 \item vertices $\Lambda_i$ with $\lambda_i=m$ (in fact, all such vertices from $\G$ are glued 
 together so they correspond to a single vertex in $\GG$).
\end{itemize}
The graph $\GG_m$ encodes all inequalities fulfilled by the collection of variables $\big(a_{l}(i)\big)$ over 
$(l,i)$ such that $| a_{l}(i) - m | < \frac{1}{2}$.

Since $\GG_m$ is acyclic, it is possible to extend it to a linearly ordered set. Let us choose any 
such a linear extension. 
There are the following two cases:
\begin{itemize}
 \item the graph $\GG_m$ does not contain any vertex $\Lambda_\ell$; then the set of solutions which is 
 compatible with the selected linear order is a simplex with the volume $$ \frac{1}{|\indeks_m|!},$$
 \item the graph $\GG_m$ contains a vertex $\Lambda_\ell$; let us say that there are $p$ (respectively, $q$) 
 vertices $A_{l}(i)$ which are smaller (respectively, bigger) than $\Lambda_\ell$ with $p+q=|\indeks_m|$;
then the set of solutions which is compatible with the selected linear order is a product of two simplexes 
with the volume $$ \frac{1}{2^{p+	q} p! q!}.$$
\end{itemize}
Note that the simplex obtained by choosing a linear order has a smaller volume than $\Cont_m$, so that the above cases give us a lower bound. Now this finishes the proof that the first condition implies the third one.

The third condition trivially implies the fourth condition.

Assume that $(b+I) \cap \Cont  \neq \emptyset$.  Let $a$ be any element of this set.
The system of inequalities \eqref{eq:system} has a particularly nice form: if $a$ is a solution then also 
$\operatorname{round}(a)$ is a solution, where $\operatorname{round}$ denotes the (coordinate-wise) 
rounding of a real number to the closest integer. On the other hand $\operatorname{round}(a)=b$ 
therefore $b\in\Disc$ which finishes the proof that the 
fourth condition implies the first condition.
\end{proof}

\end{lemma}

\subsection{Proof of Proposition \ref{lem:firstrowbound}}
\label{subsec:proof-of-firstrow}

\begin{proof}[Proof of Proposition \ref{lem:firstrowbound}]
For $x\in\Z$ (respectively, $x\in\R$) we denote by $\Disc^x\subset\Z^{d-1}$ (respectively, by 
$\Cont^x\subset\R^{d-1}$) the set of integer (respectively, real) solutions of the system of 
inequalities \eqref{eq:system} over variables $a_{l}(i)$, $(l,i)\in\indeks$, $(l,i)\neq (k,1)$, subject to 
the additional requirement that $a_{k}(1)=x$.

With respect to the subsets of $\R^{d-1}$ we denote by $\vol_{d-1}$ the usual Lebesgue volume while 
with respect to the subsets of $\Z^{d-1}$ we denote by $\vol$ the counting measure.

Now we fix $x\in \Z$. Lemma \ref{lem:disc-vs-cont} implies that
$$\int_{|x-y|<\frac{1}{2}} \vol_{d-1} \Cont^y \ dy = 
\sum_{b \in \Disc^x} \vol_d \big[ (b+I) \cap \Cont \big] \geq 
C\ \vol \Disc^x.
$$
It follows that there exists some $y$ such that
\begin{equation}
\label{eq:szacA}
 \vol_{d-1} \Cont^y  \geq 
C\ \vol \Disc^x. 
\end{equation}

It is a simple exercise to check that for $x_0\in\{\lambda_1,\lambda_{n+1-k}\}$ the set $\Cont^{x_0}$ is 
nonempty. Let us select the value of $x_0$ for which
$$ |x_0-y| \geq \frac{\lambda_1 - \lambda_{n+1-k}}{2}$$
and let us fix some $a\in\Cont^{x_0}$.

Under the obvious identifications $a\in \Cont^{x_0} \subset \Cont \subset \R^{d}$
and $\Cont^y \subset \Cont \subset \R^d$ we can consider the convex cone having $a$ as the vertex and 
$\Cont^y$ as the base. Clearly, $\Cont$ as a convex set contains this cone. It follows that for 
$t=(1-\alpha) x_0 + \alpha y$, with $0<\alpha<1$ we have
$$ \vol_{d-1} \Cont^t \geq \alpha^{d-1} \vol_{d-1}\Cont^y $$
hence 
\begin{equation}
\label{eq:szacB}
 \vol_d \Cont = \int_z \vol_{d-1} \Cont^z \ dz \geq 
\frac{\lambda_1-\lambda_{n+1-k}}{2 d} 
\vol_{d-1} \Cont^y.
 \end{equation}

Lemma \ref{lem:disc-vs-cont} shows that 
\begin{equation}
\label{eq:szacC}
 \vol \Disc \geq \vol_{d} \Cont.
 \end{equation}

Inequalities \eqref{eq:szacA}, \eqref{eq:szacB}, \eqref{eq:szacC} imply that
$$ P\big(a_k(S)=x) = \frac{\vol \Disc^x}{\vol \Disc} \leq \frac{\text{Const}}{\lambda_1-\lambda_{n+1-k}}.$$
\end{proof}

\section{Proof of the main result}
\label{sec:proof-main-result}

\newcommand{\lambdaB}{\bar{\lambda}}
\newcommand{\muB}{\bar{\mu}}
\newcommand{\nuB}{\bar{\nu}}

\begin{proof}[Proof of Theorem \ref{theorem:main-theorem}]
For a weight $\lambda=(\lambda_1,\dots,\lambda_n)$ we denote by
$\lambdaB = ( -\lambda_n, \dots, -\lambda_1)$ the weight corresponding to the 
contragredient representation. Since $d_\lambda=d_{\lambdaB}$ and 
$c_{\lambda,\mu}^{\nu} = c_{\lambdaB,\muB}^{\nuB}$ 
therefore the inequality \eqref{eq:main-gln} holds for $\lambda,\mu,\nu$ if and only if it 
holds for $\lambdaB,\muB,\nuB$.

Let $\lambda,\mu$ be fixed. By the pigeon hole principle, there exist $i,j\in \{1,\ldots ,n-1\}$ such that
\begin{align*} 
\lambda_i-\lambda_{i+1}\geq &\frac{\lambda_1-\lambda_n}{n-1}, \\
\mu_j-\mu_{j+1}\geq & \frac{\mu_1-\mu_n}{n-1}.
\end{align*}

For $i'=n-i$ and $j'=n-j$ we have analogous inequalities
\begin{align*} \lambdaB_{i'}-\lambdaB_{i'+1}\geq &\frac{\lambdaB_1-\lambdaB_n}{n-1} \\
\muB_{j'}-\muB_{j'+1}\geq & \frac{\muB_1-\muB_n}{n-1}.
\end{align*}

Since $(i+j)+(i'+j')=2n$, 
at least one of the following is true: $ i+j \leq n$ or $i'+j'\leq n$.
Therefore, without loss of generality
we will assume that $i+j \leq n$; if this is not the case, simply replace 
$\lambda, \mu, \nu$ by $\lambdaB, \muB, \nuB$. 

Let $A$ and $B$ be as in Lemma \ref{lem:first-row-A}.
Equation \eqref{eq:combinatorial-B}  implies that 
$$  P(\nu_1 = x)  \leq  \sum_{\substack{k,l\geq 1, \\k+l=n+1}}  P\big( a_k + b_l = x\big) $$
thus it is enough to find appropriate bounds for the distribution of the sum $a_k + b_l$ 
for each choice of $k$ and $l$ separately. The latter distribution is a convolution of two probability 
measures, thus
$$ P\big( a_k + b_l = x\big) \leq  
\min\left( \max_{z} P\big( a_k=z \big),\; \max_{z} P\big( b_l=z \big)  \right) 
$$
and it is enough to show that there is such a bound for $a_k$ or for $b_l$. Clearly,
$$ n+1-k\geq i+1  \quad \vee \quad n+1-l\geq j+1  $$
(otherwise $n+1=2n+2-(k+l)\leq i+j$ would contradict that $i+j\leq n$).
We will investigate these two cases separately.

In the first case,  
$$ \lambda_1 -\lambda_{n+1-k} \geq \lambda_i-\lambda_{i+1}\geq \frac{\lambda_1-\lambda_n}{n-1}. $$
We apply Lemma \ref{lem:firstrowbound}; in this way
$$ P\big( a_k=z ) \leq D_n \frac{1}{\lambda_1-\lambda_{n+1-k}} \leq
 D_n  \ \frac{n-1}{\lambda_1-\lambda_{n}}.$$

In the second case, 
$$ \mu_1 -\mu_{n+1-l} \geq \mu_j-\mu_{j+1}\geq \frac{\mu_1-\mu_n}{n-1}. $$
We apply Lemma \ref{lem:firstrowbound} for diagram $\lambda':=\mu$ and $k'=l$, in this way
$$ P\big( b_l=z ) \leq D_n \frac{1}{\mu_1-\mu_{n+1-l}} \leq
 D_n  \ \frac{n-1}{\mu_1-\mu_{n}}.$$
This completes the proof.
\end{proof}

\section{Saturation of the bound}
\label{sec:optimal}

Here we show that our bound is saturated in some natural sense. 

\begin{proposition}
For each $n$,
there exist two sequences  $(\lambda_N)$, $(\mu_N)$ of irreducible representations of $\GL(n)$ (respectively, $\SL(n)$) which tend to infinity
with the property that the inequality \eqref{eq:main-gln} of Theorem \ref{theorem:main-theorem} (respectively, inequality \eqref{eq:main-sln} of Corollary \ref{cor:main-theorem-SLn}) is saturated
up to a multiplicative constant that depends only on $n$ and not on $N$.
\end{proposition}
\begin{proof}
Take $\lambda = (N,0,\dots,0)$ and $\mu = (M,0,\dots,0)$. 
Then it is clear from Littlewood-Richardson rule that all the $\nu$ for which there is a non-zero probability $P_{\lambda,\mu}$
are of the form
$$(A,B,0,\dots,0)$$
with the constraints that $A\geq B\geq 0$, $A+B=N+M$, $A\geq \max(N,M)$.
There are $\min(N,M)$ choices.
By pigeon hole principle, at least one of these weights has a probability at least
$$ \frac{1}{\min(N,M)}$$
which is comparable to
the bound obtained in our 
Corollary \ref{cor:main-theorem-SLn}, and therefore also saturates the bound for the
main Theorem
\ref{theorem:main-theorem}.
Note that it follows from the proof that the Littlewood-Richardson coefficients appearing in this proof can not be large.
As a matter of fact, one can prove that they are all equal to $1$ in this case (but we do not need it in order to complete the proof).
\end{proof}

The above proposition shows that, for example, 
if we wanted, for a given $N$, the following inequality
$$P_{\lambda,\mu}(\nu) \leq  C_n \left(\frac{1}{\lambda_1-\lambda_n}+\frac{1}{\mu_1-\mu_n} \right)^\alpha$$
to be true for all $\mu, \nu$, then necessarily, $\alpha\leq 1$, and actually $\alpha=1$  is the best possible
constant.

Note that if the quantifier of Theorem \ref{theorem:main-theorem} is not on \emph{all}
choices of $\mu,\nu$ but just on some nice (possibly infinite) sets of pairs,
then it is possible to obtain much better estimates. 

As a first example, if in $\GL (3)$, one takes the collection $\mu_N=\nu_N=(2N,N,0)$,
it is easy to see that the largest dimension of a Littlewood-Richardson factor that can occur in 
$\mu_n\otimes \nu_n$ is at most of order $N^3$, which is less than $N^6$.
However if one in addition allows Littlewood-Richardson coefficients, then one obtains $N^5$.
Here we still saturate Theorem
\ref{theorem:main-theorem}
but not Corollary \ref{cor:main-theorem-SLn} any more.

As a second example, if one takes in $\GL(4)$ the sequence 
 $\mu_N=\nu_N=(3N,2N,N,0)$, one can see that 
 the largest dimension of a Littlewood-Richardson summand that can occur in 
$\mu_n\otimes \nu_n$ is at most of order $N^6$, which is less than $N^{12}$.
And if one in addition allows Littlewood-Richardson coefficients, then one obtains $N^9$.
Here, we are away from saturation both for Theorem
\ref{theorem:main-theorem}
and for Corollary \ref{cor:main-theorem-SLn}.

\section*{Acknowledgments}

B.C.'s research was supported by an NSERC Discovery grant and an ERA at the University of Ottawa.
He wishes to thank the organizers of the EPSRC Symposium Workshop \emph{``Interacting particle systems, 
growth models and random matrices''}, as well as Chungbuk National University and RIMS for their 
hospitality and the opportunity to meet with coworkers and make critical progress on the project.
He also thanks Ebrahim Samei for enlightening discussions.

H.H.L.'s research was supported by Basic Science Research Program through the National Research Foundation of Korea (NRF) funded by the Ministry of Education, Science and Technology (2012R1A1A2005963).

In the initial phase of research, P.\'S.~was a holder of a fellowship of \emph{Alexander von Humboldt-Stiftung}.
P.\'S.'s research has been supported by a grant of \emph{Deutsche Forschungsgemeinschaft} (SN 101/1-1).

\bibliographystyle{alpha}
\bibliography{dimsun}

\end{document}